\documentclass{amsart}

\usepackage{amsmath,amssymb,color}

\def\oblong{\square}

\numberwithin{equation}{section}

\newtheorem{theorem}[equation]{Theorem}
\newtheorem{lem}[equation]{Lemma}
\newtheorem{sch}[equation]{Scholium}
\newtheorem{cor}[equation]{Corollary}
\newtheorem{proposition}[equation]{Proposition}

\theoremstyle{definition}
\newtheorem{example}[equation]{Example}
\newtheorem{defn}[equation]{Definition}
\newtheorem{rem}[equation]{Remark}

\newcommand{\thmref}[1]{Theorem~\ref{#1}}
\newcommand{\propref}[1]{Proposition~\ref{#1}}
\newcommand{\lemref}[1]{Lemma~\ref{#1}}

\newcommand{\defref}[1]{Definition~\ref{#1}}

\newcommand{\exref}[1]{Example~\ref{#1}}

\newcommand{\secref}[1]{Section~\ref{#1}}

\newcommand\aut{\operatorname{\mathsf{Aut}}}

\newcommand{\bded}{{\mathfrak{BA}}}
\newcommand{\bdd}{{\mathfrak B}}

\newcommand{\C}{\mathbb{C}}

\newcommand\ch{\check}

\newcommand\fnst{{\mathfrak A}}
\newcommand\ftry{{\mathfrak F}}
\newcommand\full{{\mathfrak W}}

\renewcommand\L{{\mathsf L}}
\newcommand\LL{{\mathfrak L}}

\newcommand\M{{\mathsf M}}
\newcommand\mthr{{\mathfrak M}}

\newcommand{\mapstoto}{\mathop{\,\mapstochar\relbar\joinrel\relbar\joinrel\rightsquigarrow\,}}

\renewcommand\P{{\mathbf P}}
\def\pair<#1>{{\langle\!\langle}#1{\rangle\!\rangle}}
\newcommand{\prq}{\preccurlyeq}

\renewcommand\S{{\mathsf S}}
\newcommand\setsuch[2]{{\left\{#1\left|\,#2\right\}\right.}}

\newcommand\supp{\operatorname{\mathsf{supp}}}
\newcommand\sym[1]{{\mathsf{Sym}(#1)}}

\newcommand\wt{\widetilde}

\newcommand{\Z}{\mathbb{Z}}

\begin{document}

\title{On amenability of automata groups}

\author[L. Bartholdi]{Laurent Bartholdi}
\address{Institut de Math\'ematiques B, \'Ecole Polytechnique F\'ed\'erale de Lausanne, CH-1015 Lausanne, Switzerland}
\email{laurent.bartholdi@gmail.com}

\author[V. A. Kaimanovich]{Vadim A. Kaimanovich}
\address{Mathematics, Jacobs University Bremen, Campus Ring 1, D-28759, Bremen, Germany}
\email{v.kaimanovich@jacobs-university.de}

\author[V. V. Nekrashevych]{Volodymyr V. Nekrashevych}
\address{Department of Mathematics, Texas A\&M University, College Station, TX 77843-3368, USA}
\email{nekrash@math.tamu.edu}

\begin{abstract}
We show that the group of bounded automatic automorphisms of a rooted tree is
amenable, which implies amenability of numerous classes of groups generated by
finite automata. The proof is based on reducing the problem to showing
amenability just of a certain explicit family of groups (``Mother groups'')
which is done by analyzing the asymptotic properties of random walks on these
groups.
\end{abstract}

\date{\today}

\maketitle

\section*{Introduction}

Since the definition of amenability of groups by von Neumann, many attempts
were made to understand amenability and to describe it in various ways. The
class of countable amenable groups is, from the analytical point of view, the
most natural extension of the class of finite groups. Namely, according to the
original definition of von Neumann \cite{vonNeumann29} these are the groups
which admit an invariant mean (a finitely additive probability measure). An
amenable group does not contain non-abelian free subgroups. However, the
converse is not true, and, in spite of existence of numerous geometric or
analytic criteria of amenability (Tarski, F\o lner, Reiter, Kesten, etc.),
there is no satisfactory ``algebraic'' description of the class of amenable
groups. From this point of view, essentially new examples of amenable and
non-amenable groups are still of great interest.

It was proved already by von Neumann that the class of amenable groups is
closed under passing to subgroups, quotients, group extensions and inductive
limits. Therefore, starting from ``obviously'' amenable groups (which are
finite groups and the infinite cyclic group), one can construct many examples
of amenable groups. The groups obtained in this way are called
\emph{elementary amenable groups}, following Day~\cite{Day57}.

It was an open question for a long time whether every amenable group is
elementary amenable. The first example of an amenable but not elementary
amenable group is the group of intermediate growth found by
Grigorchuk~\cite{Grigorchuk80,Grigorchuk85} (every group of subexponential
growth is amenable by F\o lner's criterion). Later, a finitely presented
amenable extension of the Grigorchuk group was constructed
in~\cite{Grigorchuk98}.

Groups of subexponential growth can also be considered as ``obviously''
amenable. Therefore, a natural goal
(see~\cite{Grigorchuk98,Ceccherini-Grigorchuk-delaHarpe99}) is to find
amenable groups, which are not \emph{subexponentially elementary}, i.e., can
not be obtained from the groups of subexponential growth by the aforementioned
amenability preserving operations.

The first example of such a group is the \emph{iterated monodromy group} of
the polynomial $z^2-1$ known as the \emph{Basilica group}. It was shown
in~\cite{Grigorchuk-Zuk02a} that it does not belong to the class of
subexponentially elementary groups, whereas it was proved
in~\cite{Bartholdi-Virag05} that the Basilica group is amenable.

The aim of the present paper is to establish amenability of a vast
class of groups generated by finite automata. Namely,

\medskip

\textbf{Main Result.} \emph{Any group generated by a finite bounded
  automaton is amenable}.

\medskip

The class of groups generated by bounded automata was defined by Sidki
in~\cite{Sidki00} (see \cite{Bondarenko-Nekrashevych03} for an interpretation
of these groups in terms of fractal geometry). Most of the well-studied
examples of groups of finite automata belong to this class. In particular, it
contains the Grigorchuk group, the Gupta--Sidki group, the Basilica group, all
iterated monodromy groups of postcritically finite polynomials, and many other
examples (see \secref{sec:examples} for more details). For most of them
(except for the situation when the group happens to have subexponential
growth) our proof is the only proof of amenability known so far.

Note that the groups generated by bounded automata form a subclass of the
class of \emph{contracting self-similar groups}
(see~\cite{Bondarenko-Nekrashevych03,Nekrashevych05}). It is still an open
question whether all contracting groups are amenable.

\medskip

Any group generated by a bounded automaton is contained in the countable group
$\bded$ of \emph{all} bounded automatic automorphisms of a rooted homogeneous
tree, and it is amenability of the latter that we actually establish
(\thmref{thm:main}). Our proof is based on two ideas. First we reduce the
question about amenability of $\bded$ to that about amenability just of a
certain special family of groups which we call \emph{Mother groups}
(\thmref{thm:reduc}). Then we deduce amenability of these groups from an
analysis of the asymptotic properties of \emph{random walks} on them
(\thmref{thm:amen}). Namely, we show, by applying a self-similarity argument,
that the growth of the entropy of the $n$-fold convolutions of a certain
probability measure is sublinear, which, by the general entropy theory (see
\cite{Kaimanovich-Vershik83}), implies amenability. Therefore, our proof
ultimately uses Reiter's characterization of amenability: we construct a
sequence of approximately invariant measures on the group as the convolution
powers of a certain finitely supported one. A constructive version of this
argument based on entropy estimates yields explicit bounds for the return and
isoperimetric profiles on the Mother groups (\thmref{th:profile}). On the
other hand, we do not obtain any explicit description of the F\o lner sets.

\medskip

The paper has the following structure. In \secref{sec:main} we formulate the
main result and give a number of examples of its applications. The background
on bounded automata is discussed in \secref{sec:bounded}. In \secref{sec:fg}
we reduce the problem to amenability of Mother groups, which is established in
\secref{sec:amen mother} by an analysis of random walks on these groups.
Finally, we relegate certain auxiliary estimates of the entropy of
convolutions on general countable groups to the Appendix.

\medskip

The authors express their debt and gratitude to B\'alint Vir\'ag, who
generously contributed valuable insight to this paper.


\section{Statement of the main result} \label{sec:main}

\subsection{Decomposition of tree automorphisms}

Let $X$ be a finite set called the \emph{alphabet}. The associated
\emph{homogeneous rooted tree} $T=T(X)$ is the (right) Cayley graph of the
free monoid $X^*$ (so that one connects $w$ to $wx$ by an edge for all $w\in
X^*,x\in X$). Each vertex $w\in T\cong X^*$ is the root of the subtree $T_w$
which consists of all the words beginning with $w$. The map $w'\mapsto ww'$
provides then a canonical identification of the trees $T$ and $T_w$.

Let us denote by $\full=\full(X)=\aut(T)$ the \emph{full automorphism group}
of the tree $T$. Any automorphism $\alpha\in\full$ obviously preserves the
first level of $T$, i.e., determines a permutation
$\sigma=\sigma_\alpha\in\sym X$. Thus, any subtree $T_x$, for $x\in
X$, is mapped by
$\alpha$ onto the subtree $T_{\sigma(x)}$, which, in view of the canonical
identification of both $T_x$ and $T_{\sigma(x)}$ with $T$, gives rise to an
automorphism $\alpha_x\in\full$. Conversely, any set of data consisting of
automorphisms $\alpha_x\in\full$ for all $x\in X$ and a permutation
$\sigma\in\sym X$
determines in the above way an automorphism of $T$. Thus, we have a one-to-one
correspondence
\begin{equation} \label{eq:dec}
\alpha\mapsto\pair<\alpha_x>_{x\in X}\sigma_\alpha
\end{equation}
(called \emph{decomposition}) between $\full$ and $\full^X\times\sym
X$. We
shall omit $\sigma_\alpha$ in this notation if it is the identity
permutation. In terms of this decomposition the group multiplication in
$\full$ takes the form
\[
\pair<\alpha_x>\sigma_\alpha\cdot \pair<\beta_x>\sigma_\beta=
\pair<\alpha_x\beta_{\sigma_\alpha(x)}>\sigma_\alpha\sigma_\beta,
\]
which means that decomposition \eqref{eq:dec} is in fact a group isomorphism
between $\full$ and the \emph{permutational wreath product} $\full\wr\sym
X=\full^X\rtimes\sym X$. We shall often identify $\full$ with $\full\wr\sym X$
by the decomposition isomorphism \eqref{eq:dec}, writing
$\alpha=\pair<\alpha_x>\sigma_\alpha$, especially in recursive definitions of
automorphisms of the tree~$T$. See \cite{Bartholdi-Grigorchuk00} or
\cite[Section~2.6]{Nekrashevych05} for more on recursions of this kind and
\exref{ex:basilica} for a more detailed description of this procedure for a
concrete group.

\subsection{Generalized permutation matrices} \label{sec:gen perm}

It will also be convenient to use the matrix notation by presenting an element
$\alpha=\pair<\alpha_x>\sigma$ as a \emph{generalized permutation matrix}
$M=M^\alpha$ of order $|X|$ with entries
\[
M_{xy}=\left\{\begin{array}{ll}\alpha_x &\text{if $y=\sigma(x)$,}\\
0 & \text{otherwise.}\end{array}\right.
\]
We identify in this way the group $\full\wr\sym X$ with a subgroup of the
matrix algebra $\M_{|X|}(\C[\full])$ over the group ring of the group $\full$.
It is easy to see that this identification is actually a group isomorphism.

More generally, given an arbitrary group $G$, we shall denote by
$$
\sym {X;G}:= G\wr\sym X=G^X\rtimes\sym X
$$
the \emph{group of generalized permutation matrices} of order $|X|$ with
non-zero entries from the group $G$. Obviously, application of the
\emph{augmentation map} (which consists in replacing all group elements
with~1) to a generalized permutation matrix yields a usual permutation matrix,
which corresponds to the natural projection of $\sym {X;G} \cong
G^X\rtimes\sym X$ onto $\sym X$.

\subsection{Automatic and bounded automorphisms} \label{sec:bdd}

Recall that given an automorphism $\alpha\in\full$ any symbol $x\in X$
determines an associated automorphism $\alpha_x\in\full$ by decomposition
\eqref{eq:dec}. In the same way such an automorphism $\alpha_w\in\full$ (the
\emph{state} of $\alpha$ at the point $w$) can be defined for an arbitrary
point $w\in T\cong X^*$, by restricting the automorphism $\alpha$ to the
subtree $T_w$ with the subsequent identification of both $T_w$ and its image
$\alpha(T_w)=T_{\alpha(w)}$ with $T$. Equivalently, $\alpha_w$ can be obtained
from iterating decomposition \eqref{eq:dec}, see \exref{ex:basilica} and the
proof of \thmref{thm:reduc}.

If the \emph{set of states} of $\alpha$
$$
\S(\alpha)=\setsuch{\alpha_w}{w\in T}\subset\full
$$
is finite, then the automorphism $\alpha$ is called \emph{automatic}. The set
of all automatic automorphisms of the tree $T$ forms a countable subgroup
$\fnst=\fnst(X)$ of $\full=\full(X)$ (see \secref{sec:bounded} for more
details).

An automorphism $\alpha$ is called \emph{bounded} if the sets $\setsuch{w\in
X^n}{\alpha_w\neq1}$ have uniformly bounded cardinalities over all $n$. The
set of all bounded automorphisms forms a subgroup $\bdd=\bdd(X)$ of
$\full=\full(X)$. We denote by $\bded=\bded(X)=\bdd(X)\cap\fnst(X)$ the
\emph{group of all bounded automatic automorphisms} of the homogeneous rooted
tree $T$.

\medskip

\noindent We can now formulate the main result of the paper

\begin{theorem}\label{thm:main}
The group $\bded(X)$ is amenable for any finite set $X$.
\end{theorem}

\subsection{Examples} \label{sec:examples}

In the rest of this Section we describe some interesting finitely generated
subgroups of $\bded$, amenability of which follows from
Theorem~\ref{thm:main}. We define the generators of these groups by their
decomposition \eqref{eq:dec}.

\begin{example} \label{ex:basilica}
Let $|X|=2$, denote by $\sigma$ the non-trivial element of $\sym X$, and
define the automorphisms $a,b$ recursively by the relations
\[
a=\pair<b,1>,\qquad b=\pair<a,1>\sigma,
\]
or, in matrix terms,
\[
M^a=\left(\begin{array}{cc} b & 0\\ 0 &
  1\end{array}\right),\qquad
  M^b=\left(\begin{array}{cc} 0 & a\\ 1 & 0\end{array}\right).
\]
More precisely, application of the augmentation map to the above generalized
permutation matrices yields the usual permutation matrices of order 2 which
describe the action of $a$ and $b$ on the first level $X$ of the tree $T$.
Substitution of $M^a$ for $a$ and $M^b$ for $b$ gives the order 4 generalized
permutation matrices
$$
\begin{pmatrix}
  0 & a & 0 & 0 \\
  1 & 0 & 0 & 0 \\
  0 & 0 & 1 & 0 \\
  0 & 0 & 0 & 1 \\
\end{pmatrix}
\;, \;
\begin{pmatrix}
  0 & 0 & b & 0 \\
  0 & 0 & 0 & 1 \\
  1 & 0 & 0 & 0 \\
  0 & 1 & 0 & 0 \\
\end{pmatrix} \;,
$$
which, after applying the augmentation map, give rise to the usual order 4
permutation matrices describing the action of $a$ and $b$, respectively, on
the second level $X^2$ of the tree $T$ which extends the action on $X$. By
iterating this substitution once again we obtain the action of $a$ and $b$ by
permutations on $X^3$, and so on, so that in the limit we obtain automorphisms
of the full tree $T$. Note that the entries of the arising matrices are the
states of these automorphisms, and therefore we can immediately see that both
$a$ and $b$ are automatic and bounded.

The group $G=\langle a,b\rangle$ is called the \emph{Basilica group} (because
it is the iterated monodromy group of the \emph{Basilica polynomial} $z^2-1$),
it is contained in $\bded$, and it is amenable~\cite{Bartholdi-Virag05} but not
``subexponentially elementary amenable''~\cite{Grigorchuk-Zuk02a}.
\end{example}

\begin{example}
More generally, let $f(z)\in\C[z]$ be a \emph{postcritically finite} complex
polynomial, i.e., such that for every critical point $c$ of $f(z)$ the orbit
$\setsuch{f^{n}(c)}{n\ge 1}$ is finite. Let $P$ be the union of the orbits of
all the critical points of $f$. Given a point $t\in\C\setminus P$ the
fundamental group $\pi_1(\C\setminus P, t)$ naturally acts by monodromy on the
\emph{preimage tree} $T$, whose vertex set consists of all the pairs
$\{(f^{-n}(t),n)\}_{n\ge0}$ with edges joining $(\tau,n)$ and $(f(\tau),n-1)$.
The resulting group of automorphisms of the tree $T$ is called the
\emph{iterated monodromy group} of the polynomial $f$. For more on iterated
monodromy groups see~\cite{Nekrashevych05}. In particular, it is proved
in~\cite[Chapter~6]{Nekrashevych05} that iterated monodromy groups of
postcritically finite polynomials are subgroups of $\bded$, hence they are
amenable by Theorem~\ref{thm:main}.
\end{example}

\begin{example}
Let $X$ and $\sigma$ be as in \exref{ex:basilica}, and define the
automorphisms $a,b$ by putting
\[
a=\pair<1,a>\sigma,\quad b=\pair<1,b^{-1}>\sigma\;,
\]
i.e.,
$$
M^a = \left(\begin{array}{cc}
  0 & 1\\ a & 0\end{array}\right) \;, \quad M^b = \left(\begin{array}{cc}
  0 & 1\\ b^{-1} & 0\end{array}\right)\;.
$$
The group $G=\langle a,b\rangle$ determined by the above presentation is
contained in $\bded$, and it was studied by Brunner, Sidki and Vieira
in~\cite{Brunner-Sidki-Vieira99}. Later da~Silva showed in her
thesis~\cite{daSilva01} that $G$ does not contain any non-abelian free
subgroups. Since $G$ is amenable by Theorem~\ref{thm:main}, we obtain another
proof of that result.
\end{example}

\begin{example}
Let $\sigma\in\sym X$ be a cyclic permutation of the alphabet $X$, and choose
$\varepsilon_2,\dots,\varepsilon_d$ from the cyclic group $\Z/d$, where
$d=|X|$. Let $G=\langle a,b\rangle$ be the group generated by two order $d$
elements determined by the decompositions
$$
a=\pair<1,1,\dots,1>\sigma \;, \quad
b=\pair<b,a^{\varepsilon_2},\dots,a^{\varepsilon_d}> \;.
$$
In particular, if $d=3$ and $(\varepsilon_i)=(1,-1)$ then $G$ is the infinite
$2$-generated $3$-group studied by Gupta and Sidki in~\cite{Gupta-Sidki83},
and if $d=3$ and $(\varepsilon_i)=(1,0)$ then $G$ is the group of intermediate
growth studied by Fabrykowski and Gupta in~\cite{Fabrykowski-Gupta91}. This
family of groups was called \emph{GGS} groups (referring to Grigorchuk, Gupta
and Sidki) by Baumslag \cite{Baumslag93}. They are all subgroups of $\bded$.
\end{example}

\begin{example}
Let $A$ be a subgroup of $\sym X$. We shall consider two embeddings
$\theta_1,\theta_2$ of $A$ into $\full$ determined by the decompositions
$$
\theta_1(a)=\pair<1,1,\dots,1>a \;,\quad
\theta_2(a)=\pair<\theta_1(a),\theta_2(a),1,\dots,1> \;,
$$
respectively, and then set $G=\langle\theta_1(A),\theta_2(A)\rangle$. These
groups were considered by Neumann in~\cite{Neumann86} to answer some questions
of ``largeness'' formulated by Edjvet and Pride, and more recently by the
first author~\cite{Bartholdi03} to construct groups of exponential word growth
for which the infimum of the growth rates is $1$ (also see \cite{Wilson04}).
All of these groups are subgroups of $\bded$.
\end{example}

\section{Bounded automata} \label{sec:bounded}

In this Section we recall some standard facts about automata,
see~\cite{Grigorchuk-Nekrashevich-Sushchanskii00} and~\cite{Sidki00} for
further details.

\subsection{Automata and automorphisms}

\begin{defn}
An \emph{automaton} $\Pi$ is a map of the product $X\times Q$ of two sets to
itself. One of these sets $X$ is called the \emph{alphabet} and the other one
$Q$ is called the \emph{state space} of the automaton. If $Q$ is finite then the
automaton is called \emph{finite}. The components
$$
\Pi_\oblong:X\times Q\to X \;, \quad \Pi_\bullet:X\times Q\to Q
$$
of the map $\Pi$ are called the \emph{output} and the \emph{transition}
functions of the automaton, respectively. An automaton $\Pi$ is
\emph{invertible} if $\Pi_\oblong(\cdot,q)$ is a bijection $X\to X$ for all
$q\in Q$. We shall always impose that condition.
\end{defn}

We interpret an automaton $\Pi$ as a machine which, being in state $q$ and
reading an input letter $x$, goes to state $\Pi_\bullet(x, q)$ and outputs the
letter $\Pi_\oblong(x,q)$. In this way it can also process words, which gives
rise to the automaton $\Pi^*$ with extended alphabet $X^*$ and same state
space $Q$. Its output and transition functions $\Pi_\oblong^*,\Pi_\bullet^*$
are extensions of the respective original functions $\Pi_\oblong,\Pi_\bullet$
and are defined recursively as
\begin{equation} \label{eq:*}
\begin{aligned}
 \Pi^*_\oblong(x_1x_2\dots x_n,q) &= \Pi_\oblong(x_1,q) \Pi^*_\oblong\left(x_2\dots x_n,\Pi_\bullet(x_1,q)\right) \;, \\
 \Pi_\bullet^*(x_1x_2\dots x_n,q) &= \Pi_\bullet^*(x_2\dots x_n,\Pi_\bullet(x_1,q)) \;.
\end{aligned}
\end{equation}
Invertibility of $\Pi$ implies invertibility of the extended automaton $\Pi^*$
as well, whence

\begin{defn} \label{def:autom}
A state $q$ of an automaton $\Pi$ determines an automorphism
$\Pi^*_\oblong(\cdot,q)$ of the tree $T(X)$ over its alphabet $X$. Such an
automorphism is called \emph{automatic}. Below we shall always identify the
state $q$ with the associated automorphism $\Pi^*_\oblong(\cdot,q)$, i.e., we
shall assume $Q\subset\full$.
\end{defn}

\begin{proposition} \label{prop:equiv}
An automorphism $\alpha\in\full$ is automatic in the sense of
\defref{def:autom} if and only if it is automatic in the sense of the
definition given in \secref{sec:bdd}, i.e., if and only if its set of states
$\S(\alpha)$ is finite.
\end{proposition}

\begin{proof}
If $\alpha=\Pi^*_\oblong(\cdot,q)$ is automatic, then $\S(\alpha)$ is
precisely the set of states of the automaton $\Pi$ attainable from the state
$q$.

Conversely, given an automorphism $\alpha\in\full$, for any $q\in\S(\alpha)$
the associated decomposition $q=\pair<q_x>_{x\in X}\sigma_q$ obviously
contains only elements of $\S(\alpha)$, so that we have maps
$$
\Pi_\oblong(x,q)=\sigma_q(x) \;, \quad  \Pi_\bullet(x,q)=q_x \;,
$$
which, if the set $\S(\alpha)$ is finite, determine an automaton $\Pi$ with
alphabet $X$ and state space $\S(\alpha)$ with the property that
$\Pi^*_\oblong(\cdot,q)=q$ for all $q\in\S(\alpha)$.
\end{proof}

\subsection{Growth of automorphisms}

\begin{defn}
The \emph{growth function} $\Gamma_\alpha$ of an automorphism $\alpha$ is
defined as the growth function of the language
$$
\L(\alpha)=\setsuch{w\in X^*}{\alpha_w\neq 1}\;,
$$
i.e.,
\begin{equation} \label{eq:growth funct}
\Gamma_\alpha(n)=\left|\setsuch{w\in X^n}{\alpha_w\neq1}\right|.
\end{equation}
\end{defn}

Denote by $\bdd_d$ the set of automorphisms whose growth is bounded by a
polynomial of degree $d$, so that, in particular, $\bdd=\bdd_0$ is the set of
\emph{bounded automorphisms} introduced in \secref{sec:bdd}, and let
$\ftry=\bdd_{-1}$ be the set of \emph{finitary automorphisms}, i.e., the ones
for which the growth function \eqref{eq:growth funct} is eventually $0$
(obviously, $\ftry\subset\fnst$). Note that if $\alpha$ is automatic, then the
language $\L(\alpha)$ is regular (since it is recognized by a finite
automaton), so that in this case the growth function $\Gamma_\alpha$ is either
polynomial or exponential.

It is easy to see that the growth function is \emph{symmetric} and
\emph{subadditive} with respect to $\alpha$, i.e.,
$\Gamma_\alpha=\Gamma_{\alpha^{-1}}$ and
$\Gamma_{\alpha\beta}\le\Gamma_\alpha+\Gamma_\beta$, so all subsets $\bdd_d$
are subgroups of $\full$. The groups $\bdd_d$ do not contain non-abelian free
subgroups \cite{Sidki04}.

We shall say that an automorphism $\alpha\in\full$ is \emph{directed} if there
exists a word $w_0\in X^l$ such that $\alpha_{w_0}=\alpha$, and all the other
states $\alpha_w$ with $w\in X^l$ are finitary. The smallest number $l$ with
this property is called the \emph{period} of $\alpha$.

The following description of the group $\bded=\bdd\cap\fnst$ follows
from~\cite[Corollary~14]{Sidki00}.

\begin{proposition} \label{pr:strbndd}
An automatic automorphism $\alpha$ is bounded if and only if it is either
finitary or there exists an integer $m$ such that all non-finitary states
$\alpha_w$ with $w\in X^m$ are directed.
\end{proposition}

\begin{defn} \label{def:depth}
By using \propref{pr:strbndd} we can now define the \emph{depth} of an
arbitrary automatic bounded automorphism $\alpha$: if $\alpha$ is finitary,
then its \emph{finitary depth} is the smallest integer $m$ such that all the
states $\alpha_w,\,w\in X^m$ are trivial; otherwise the \emph{bounded depth}
of $\alpha$ is the smallest integer $m$ from \propref{pr:strbndd}.
\end{defn}


\section{Finitely generated subgroups of $\bded$ and the Mother group}\label{sec:fg}

We show in this Section that a finitely generated group of bounded
automorphisms can be put into a particularly simple form.

\subsection{The Mother group}

\begin{defn}[``Mother group'']\label{defn:mother}
  Let $X$ be a finite set with a distinguished element $o\in X$, and
  put $\overline X=X\setminus\{o\}$. Set $A=\sym X$ and
  $B=\sym{\overline X}\wr A=\sym{\overline X;A}$, and recursively
  embed the groups $A$ and $B$ into $\full(X)$ as
  $$A\ni a\mapsto(1,\dots,1)a \quad \text{ and } \quad
  B\ni b=(b_2,\dots,b_d)\sigma\mapsto(b,b_2,\dots,b_d)\sigma,$$
  assuming that $X=\{o=1,\dots,d\}$. Still in that notation, the
  matrix presentations of $a,b$ are given by
$$
M^a = \phi_A(a) \;, \quad M^b=\begin{pmatrix} b & 0 \\ 0 & \phi_B(b)
\end{pmatrix} \;,
$$
where $\phi_A(a),\phi_B(b)$ are, respectively, the permutation and the
generalized permutation matrices corresponding to $a\in A,b\in B$. Then the
\emph{Mother group} $\mthr=\mthr(X)=\langle A, B\rangle$ is the subgroup of
$\full$ generated by the finite groups $A$ and $B$.
\end{defn}

A direct verification shows that both groups $A,B$ are contained in $\bded$,
whence

\begin{proposition}
The group $\mthr(X)$ is a subgroup of $\bded(X)$.
\end{proposition}

\subsection{Embedding of finitely generated subgroups of $\bded$}

\begin{theorem}\label{thm:reduc}
Any finitely generated subgroup of $\bded(X)$ can be embedded as a subgroup
into the wreath product $\mthr(X^N)\wr\sym{X^N}$ for some integer $N$.
\end{theorem}

\begin{proof}
Let $G=\langle S\rangle$ be a finitely generated subgroup of $\bded$, and let
$\Pi$ be the automaton with alphabet $X$ and state space
$Q=\bigcup_{\alpha\in S}\S(\alpha)$ which is the union of the automata
associated with each automorphism $\alpha\in S$ (see the proof of
\propref{prop:equiv}). By boundedness, each $\S(\alpha)$ contains the identity
automorphism 1, so that $1\in Q$.

Let $F=\ftry\cap Q$ be the set of finitary elements of $Q$, let $m$ be an
integer greater than the depths of all the elements of $Q$ (see
\defref{def:depth}), and finally let $\ell$ be a common multiple of the periods
of directed automorphisms associated with non-finitary elements of $Q$.

\medskip

First we apply $m$ times decomposition \eqref{eq:dec} to the group $G$, i.e.,
embed it into the wreath product $H\wr\left(\wr^m\sym X\right)$, where
$\wr^m\sym X<\sym{X^m}$ is the automorphism group of the subtree consisting of
the first $m$ levels of the tree $T$ and $H$ is the group generated by all the
states $\alpha_w$ with $\alpha\in G$ and $w\in X^m$. Thus, $H=\langle
R\rangle$ for the subset $R=\setsuch{q_w}{q\in Q, w\in X^m}\subset Q$ of the
states of $\Pi$.

\medskip

We next replace $X$ by $X'=X^\ell$ and denote by $T'=(X')^*$ the associated
tree, which is obtained from the tree $T$ by retaining only the levels whose
numbers are multiples of $\ell$. Then $H$ is \emph{a fortiori} a group of
automatic automorphisms of $T'$. In that process, the automaton $\Pi$ is
replaced by an automaton $\Pi'$ with alphabet $X'$, but with the same state
space $Q$ as $\Pi$. Its output and transition functions are the restrictions
of the respective functions of the automaton $\Pi^*$ \eqref{eq:*}.

Let us fix a letter $o'\in X'$, a transitive cycle $\varsigma\in\sym{X'}$, and
for $x\in X'$ put $\varsigma_x=\varsigma^i$ for the unique $i\;
(\textrm{mod}\, |X'|)$ such that $x=\varsigma^i(o')$. We define an
automorphism $\delta\in\aut(T')$ via its decomposition \eqref{eq:dec} as
$\delta=\pair<\delta'_x>_{x\in X'}$ with $\delta'_x=\delta\varsigma_x^{-1}$.
In other words, the automorphism $\delta$ maps a word
$\varsigma^{i_1}(o')\varsigma^{i_2}(o')\varsigma^{i_3}(o')\ldots\varsigma^{i_n}(o')\in
T'$ to the word
\[\varsigma^{i_1}(o')\varsigma^{i_2-i_1}(o')\varsigma^{i_3-i_2}(o')\ldots
\varsigma^{i_n-i_{n-1}}(o') \;.\]

Then the $\delta$-conjugate of any automorphism
\begin{equation} \label{eq:alpha}
\alpha=\pair<\alpha'_x>_{x\in X'}\sigma\in\aut(T')
\end{equation}
is
\begin{equation} \label{eq:ad}
\alpha^\delta = \delta^{-1}\alpha\delta =
\pair<{\delta'_x}^{-1}\alpha'_x\delta'_{\sigma(x)}>_{x\in X'} \sigma =
\pair<\varsigma_x\delta^{-1}\alpha'_x\delta\varsigma^{-1}_{\sigma(x)}>_{x\in
X'} \sigma =
\pair<\varsigma_x{\alpha'_x}^\delta\varsigma^{-1}_{\sigma(x)}>_{x\in X'}
\sigma \;.
\end{equation}

By the choice of $\ell$, each $\alpha\in R$ either belongs to $F$ or else has
decomposition \eqref{eq:alpha} with the property that $\alpha'_z=\alpha$ for
precisely one letter $z=z(\alpha)\in X'$, and $\alpha'_x\in F$ whenever $x\neq
z$. In the latter case for
$\beta=\beta(\alpha)=\varsigma_z\alpha^\delta\varsigma^{-1}_{\sigma(z)}$ we
have $\beta=\pair<\beta'_x>\rho'$ with $\beta'_{o'}=\beta$, $\beta'_x\in\ftry$
for any $x\in X'\setminus\{o'\}$, and the permutation
$\rho'=\varsigma_z\sigma\varsigma^{-1}_{\sigma(z)}\in\sym{X'}$ satisfies
$\rho'(o')=o'$.

\medskip

Denote by $m'$ the maximal bounded depth of the automorphisms $\beta'_x$ from
the previous paragraph for all $x\in X'\setminus\{o'\}$ and $\alpha\in R$, and
finally enlarge once more the alphabet $X'$ to $X''=(X')^{m'}$ by putting
$o''=(o')^{m'}$. Then in the associated decomposition
$\beta=\pair<\beta''_x>_{x\in X''}\rho''$ with $\rho''\in\sym{X''}$ we have
$\rho''(o'')=o''$ and $\beta''_{o''}=\beta$. All the other automorphisms
$\beta''_x,\,x\in X''\setminus\{o''\}$ are finitary of depth at most $m'$ with
respect to the alphabet $X'$. Consequently they are finitary of depth at most
1 with respect to the alphabet $X''$, i.e., they belong to $\sym{X''}$.
Therefore, $\beta\in\mthr(X'')$. Since the auxiliary element $\varsigma$ also
belongs to $\mthr(X'')$, we conclude that the $\delta$-conjugate
$\alpha^\delta$ belongs to $\mthr(X'')$, so that the $\delta$-conjugate of the
whole group $H=\langle R\rangle$ is a subgroup of $\mthr(X'')$.
\end{proof}

\subsection{Amenability of the group $\bded$}

\thmref{thm:reduc} allows us to reduce the question about the amenability of
the groups $\bded(X)$ to the one about the amenability of the groups
$\mthr(X)\subset\bded(X)$ from \defref{defn:mother}. Further developing the
ideas from \cite{Bartholdi-Virag05} and~\cite{Kaimanovich05} we shall prove in
\secref{sec:amen mother}

\begin{theorem}\label{thm:amen}
For any finite alphabet $X$ the associated Mother group $\mthr=\mthr(X)$ is
amenable.
\end{theorem}

\begin{cor}[= Theorem~\ref{thm:main}]
The group $\bded(X)$ is amenable.
\end{cor}

\begin{proof}
To show that $\bded(X)$ is amenable, it suffices to show that all its finitely
generated subgroups are amenable. Now by \thmref{thm:reduc} such a subgroup
embeds, for a certain integer $N$, in $\mthr(X^N)\wr\sym{X^N}$, which is
amenable because $\mthr(X^N)$ is amenable.
\end{proof}

\section{Amenability of the Mother group} \label{sec:amen mother}

\subsection{Random walks on self-similar groups} \label{sec:ss}

Let $G\subset\full=\full(X)$ be a countable \emph{self-similar group}, i.e.,
such that for any $g\in G$ all the elements $g_x$ from the
decomposition $g=\pair<g_x>\sigma_g$ belong to $G$. We then have an embedding
(not an isomorphism, generally speaking!) $G\to G\wr\sym X$. In matrix terms
it becomes an embedding $g\mapsto M^g$ of the group $G$ into the group of
generalized permutation matrices $\sym {X;G}$, see \secref{sec:gen perm}. The
latter embedding extends by linearity to an algebra homomorphism
\begin{equation} \label{eq:Mm}
\mu\mapsto M^\mu=\sum\mu(g) M^g
\end{equation}
of the Banach algebra $\ell^1(G)$ into $\M_{|X|}(\ell^1(G))$.

The correspondence $\mu\mapsto M^\mu$ has a natural interpretation in terms of
\emph{random walks} on $G$, see \cite{Kaimanovich05}. Let $\mu$ be a
probability measure on $G$; then the associated random walk $(G,\mu)$ is the
Markov chain with transition probabilities $p(g,gh)=\mu(h)$, which we denote
as
$$
g \mapstoto_{h\sim\mu} gh \;.
$$
By applying the embedding $g\mapsto M^g$, it gives rise to the random walk on
the group $\sym {X;G}$ with transition probabilities
$$
M \mapstoto_{h\sim\mu} M M^h \;.
$$
Further, each of the rows of matrices from $\sym{X;G}$ performs a Markov chain
with transition probabilities
\begin{equation} \label{eq:rwidf}
R \mapstoto_{h\sim\mu} R M^h \;.
\end{equation}
Due to the definition of the group $\sym{X;G}$ the rows of the corresponding
matrices can be identified with points of the product space $G\times X$ (each
row has precisely one non-zero entry, so that it is completely described by
the value of this entry and by its position). Therefore, the latter Markov
chain can be interpreted as a Markov chain on $G\times X$ whose transition
probabilities are easily seen to be invariant with respect to the left action
of $G$ on $G\times X$. Such Markov chains are called random walks on $G$ with
\emph{internal degrees of freedom} (parameterized by $X$), for short RWIDF.
Random walks with internal degrees of freedom are described by order $|X|$
matrices $M=(M_{xy})_{x,y\in X}$ whose entries $M_{xy}$ are subprobability
measures on $G$ such that $\sum_y \|M_{xy}\|=1$ for any $x\in X$ (here
$\|\mu\|$ denotes the mass of a measure $\mu$), so that the transition
probabilities are then
\begin{equation} \label{eq:trans}
p\bigl((g,x),(gh,y)\bigr) = M_{xy}(h) \;.
\end{equation}
The projection of the RWIDF governed by $M$ to the space of degrees of
freedom $X$ is the Markov chain with transition probabilities
$p(x,y)=\|M_{xy}\|$.

Now, the interpretation promised at the beginning of this paragraph is that
the matrix describing the RWIDF \eqref{eq:rwidf} is precisely the matrix
$M^\mu$ from \eqref{eq:Mm}.

\subsection{Random walks and amenability}

The use of random walks for proving amenability of a self-similar group $G$ is
based on an idea which first appeared in \cite{Bartholdi-Virag05} and was
further developed in \cite{Kaimanovich05}.

It is well-known that amenability of a countable group $G$ is equivalent to
existence of a probability measure $\mu$ on $G$ such that it is
\emph{non-degenerate} (in the sense that its support generates $G$ as a group)
and the \emph{Poisson boundary} of the associated random walk $(G,\mu)$ is
trivial. In addition, if the measure $\mu$ has finite entropy $H(\mu)$, then
there is a quantitative criterion of triviality of the Poisson boundary: it is
equivalent to vanishing of the \emph{asymptotic entropy} $h(G,\mu)=\lim
H(\mu^n)/n$, where $\mu^n$ denotes the $n$-fold convolution of the measure
$\mu$, see \cite{Kaimanovich-Vershik83}. Thus,

\begin{theorem}[\cite{Kaimanovich-Vershik83}] \label{thm:entr}
If a countable group $G$ carries a non-degenerate probability measure $\mu$
with $h(G,\mu)=0$ then $G$ is amenable.
\end{theorem}

If the group $G$ is self-similar, then, as it was explained in \secref{sec:ss}
above, any random walk $(G,\mu)$ gives rise to a RWIDF $(G\times X,M^\mu)$.
In \cite{Bartholdi-Virag05} and \cite{Kaimanovich05} one passed then from the
RWIDF $(G\times X,M^\mu)$ to a new random walk $(G,\mu')$ by taking the
\emph{trace} of the RWIDF $(G\times X,M^\mu)$ on a single ``layer''
$G\times\{x_0\}\subset G\times X$ for an appropriately chosen letter $x_0\in
X$. The asymptotic entropy does not decrease under this passage: $h(G,\mu)\le
h(G,\mu')$. Therefore, if the measure $\mu$ is \emph{self-similar} in the
sense that $\mu'=\alpha\mu +(1-\alpha)\delta_e$ for a certain real $\alpha<1$
(here $\delta_e$ denotes the unit mass at the group identity), then
$h(G,\mu)\le\alpha h(G,\mu)$, so that the asymptotic entropy must vanish
(\emph{M\"unchhausen trick}) proving amenability of the group $G$.

In the present paper we take a different approach based on the fact that the
Mother group $\mthr$ is generated by two finite subgroups $A$ and $B$. We take
as measure $\mu$ the convolution product of the uniform measures $\mu_A$ and
$\mu_B$ on these subgroups. Then the matrix $M^\mu$ has a very special form, so
that the projection of the associated RWIDF $(\mthr\times X, M^\mu)$ to
$\mthr$ is just the random walk $(\mthr,\wt\mu)$ determined by a new measure
$\wt\mu$. The measure $\wt\mu$ is a convex combination of the idempotent
measures $\mu_A$ and $\mu_B$, so that its convolution powers are essentially
convex combinations of the convolution powers of $\mu$. We then compare the
asymptotic entropies of $\mu$ and $\wt\mu$ and use the M\"unchhausen trick in
order to deduce vanishing of the asymptotic entropy $h(\mthr,\mu)$ and to
apply \thmref{thm:entr}. Actually, we make this argument more explicit in
order to obtain a lower estimate for the return profile of $\mthr$.

\subsection{Proof of \thmref{thm:amen}}

Let us consider on $\mthr$ the probability measure
\begin{equation} \label{eq:mu}
\mu = \mu_A \mu_B \;,
\end{equation}
where $\mu_A$ and $\mu_B$ are the uniform measures on the finite subgroups $A$ and
$B$ from \defref{defn:mother}, respectively. Then the associated matrix
$M^\mu$ is
$$
M^\mu = M^{\mu_A} M^{\mu_B} = E_d \begin{pmatrix} \mu_B & 0 \\ 0 & \mu_A E_{d-1}
\end{pmatrix} \;,
$$
where $d=|X|$, and $E_d$ denotes the order $d$ matrix with entries $1/d$, so
that $M^\mu$ has identical rows with entries
$$
M^\mu_{xy} =
  \begin{cases}
    \mu_B/d  &  \text{if $y=o$} \;, \\
    \mu_A/d  & \text{otherwise} \;.
  \end{cases}
$$
It means that transition probabilities \eqref{eq:trans} of the associated
RWIDF $(\mthr\times X, M^\mu)$ do not depend on $x$, so that its projection to
$\mthr$ is just the random walk $(\mthr,\wt\mu)$ determined by the measure
$$
\wt\mu = \sum_y M^\mu_{xy} = \frac{d-1}{d} \mu_A + \frac{1}{d} \mu_B \;,
$$
whereas the projection of RWIDF $(\mthr\times X, M^\mu)$ to $X$ is the
sequence of independent $X$-valued random variables with uniform distribution
on $X$ (because all entries $M^\mu_{xy}$ have mass $1/d$). Note that these two
projections are \emph{not} independent.

Let us now compare the entropies
$$
F(n) = H(\mu^n) \;,\qquad \wt F(n) = H(\wt\mu^n)
$$
of convolution powers of the measures $\mu$ and $\wt\mu$, respectively.

First suppose that we start the RWIDF $(\mthr\times X, M^\mu)$ at time $0$
from a point $(g,x)\in\mthr\times X$. Then its time $n$ distribution is
$R(M^\mu)^n$, where $R$ denotes the vector
$(0,\dots,\delta_g,\dots,0)\in\ell_1(G)^X$ with $\delta_g$ at position $x$. By
Scholium~\ref{sch:A0}, the entropy of this distribution does not exceed the
sum of the entropies of its projections to $\mthr$ and to $X$. The projection
of $R(M^\mu)^n$ to $X$ is uniform, so its entropy is $\log d$, whereas its
projection to $\mthr$ is $\wt\mu^n$. Therefore, the entropy of the row
distribution $R(M^\mu)^n$ is at most $\wt F(n) + \log d$.

Now, the measure $\mu^n$ is the time $n$ distribution of the random walk
$(\mthr,\mu)$. As it was explained in \secref{sec:ss}, this distribution can
be identified with the time $n$ distribution of the corresponding random walk
on the group $\sym{X;\mthr}$. Again by Scholium~\ref{sch:A0}, the entropy of
the latter distribution of random matrices is at most the sum of the entropies
of all the row distributions of these matrices. The distribution of the row
parameterized by $x\in X$ is precisely $R(M^\mu)^n$ for the vector
$R=(0,\dots,\delta_e,\dots,0)\in\ell_1(G)^X$ with $\delta_e$ at position $x$;
so we have the inequality
\begin{equation} \label{eq:first}
F(n) \le d \cdot [ \wt F(n) + \log d ] = d \wt F(n) + d\log d \;.
\end{equation}
Here we interpreted the RWIDF $(\mthr\times X, M^\mu)$ as a ``row
chain'' \eqref{eq:rwidf} and used the fact that the amount of
information about a random matrix does not exceed the sum of amounts
of information about its rows.

Our next step will be to obtain a bound in the opposite direction which will
ultimately lead to vanishing of the asymptotic entropy $h(\mthr,\mu)=\lim
F(n)/n$. Since the measures $\mu_A,\mu_B$ are idempotent, the convolution power
\begin{equation} \label{eq:Hmu^n}
\wt\mu^n = \left( \frac{d-1}{d} \mu_A + \frac{1}{d} \mu_B \right)^n = \sum_{i=1}^n
p_{A,i} \mu_{A,i} + \sum_{i=1}^n p_{B,i} \mu_{B,i}
\end{equation}
is a convex combination of the alternating convolution products
$\mu_{A,i}=\mu_A\mu_B\dots$ (respectively $\mu_{B, i}=\mu_B\mu_A\dots$) of
length $i\le n$ of the measures $\mu_A$
and $\mu_B$. The probability distribution $(p_{A, i}, p_{B,
i})$ admits a simple interpretation in terms of the sequence of Bernoulli
random variables $(\xi_k)$ with distribution
$$
\P\{\xi_i=A\}=\frac{d-1}d \;, \quad \P\{\xi_i=B\}=\frac1d \;.
$$
Namely, $p_{A,i}$ (respectively $p_{B,i}$) is the probability that $\xi_1=A$ (respectively
$\xi_1=B$) and the sequence $\xi_1,\xi_2,\dots,\xi_n$ contains precisely $i$
series consisting of repetitions of the same symbol (or, equivalently, that
there are precisely $i-1$ \emph{switch times} $t$ such that
$\xi_t\neq\xi_{t+1}$ with $1\le t\le n-1$). Clearly, the probability that any
given $t$ is a switch time is $\ell=2(d-1)/d^2$, whence the expectation of the
amalgamated distribution $p_i = p_{A, i}+p_{B, i}$ is $(n-1)\ell+1$. By using
\eqref{eq:conv} it is easy to see that
$$
H(\mu_{A, i}), H(\mu_{B, i}) \le F( \lfloor i/2\rfloor + 1) \qquad \text{for
all}\; i\in\{1, \ldots, n\},
$$
where $\lfloor\cdot\rfloor$ denotes the integer part (for example, if $i$ is
even then $H(\mu_{B,i})\le H(\mu_A \mu_{B,i}\mu_B) = H(\mu_{A,i+2}) = F(i/2+1)$). Then
from \eqref{eq:Hmu^n} and \eqref{eq:entrineq} we get
$$
\wt F(n) \le \sum p_i F(\lfloor i/2\rfloor+1) + \log(2n) \;.
$$
By applying the Chebyshev inequality to the distribution $p$ (one can check
directly that its variance is linear as a function of $n$) and using the fact
that the function $F$ is monotone and subadditive (so that its values for all
integers up to $n/2$ are controlled from above just by its value at
$\left\lfloor\textstyle\frac{d-1}{d^2}n\right\rfloor$), we obtain that for any
$\epsilon>0$ and all sufficiently large $n$
\begin{equation} \label{eq:Htil}
\wt F(n) \le
F\left(\left\lfloor\left(\textstyle\frac{d-1}{d^2}+\epsilon\right)n\right\rfloor\right)
+ \log(2n).
\end{equation}
Inequalities \eqref{eq:first} and \eqref{eq:Htil} imply, after dividing by $n$
and passing to the limit, the corresponding inequalities for the asymptotic
entropies of the measures $\mu$ and $\wt\mu$:
$$
h(\mthr,\mu) \le d\, h(\mthr,\wt\mu) \;,\qquad h(\mthr,\wt\mu) \le
\textstyle\frac{d-1}{d^2}\, h(\mthr,\mu) \;,
$$
whence $h(\mthr,\mu)\le\frac{d-1}{d}h(\mthr,\mu)$, so that $h(\mthr,\mu)=0$,
and the group $\mthr$ is amenable by \thmref{thm:entr}.

\begin{rem}
Triviality of the Poisson boundary of the measure $\mu$ \eqref{eq:mu} implies
that the convolution powers $\mu^n$ satisfy the \emph{Reiter condition} of
strong convergence to left-invariance, i.e., $\|g\mu^n-\mu^n\|\to 0$ for any
$g\in\mthr$ \cite{Kaimanovich-Vershik83}. Moreover, the reflected measure
$\ch\mu=\mu_B \mu_A$ (defined by $\ch\mu(g)=\mu(g^{-1})$) has the same asymptotic
entropy as $\mu$, so that $h(\mthr,\ch\mu)$ also vanishes, and the convolution
powers $\ch\mu^n$ also satisfy the Reiter condition. This fact easily implies
that for any probability measure $\mu'$ (other than convex combinations of
$\mu_A$ or $\mu_B$ with $\delta_e$) from the $\ell^1$-closure of the algebra
generated by the measures $\mu_A$ and $\mu_B$ its convolution powers satisfy the
Reiter condition, and therefore the Poisson boundary of $\mu'$ is trivial. Are
there any measures on $\mthr$ with a non-trivial Poisson boundary?
\end{rem}

\subsection{Explicit estimates}

Inequalities \eqref{eq:first} and \eqref{eq:Htil} imply that for any
$\varepsilon>0$ the sequence of entropies $\wt F(n)$ of the \emph{symmetric}
measure $\wt\mu$ satisfies the inequality
\begin{equation} \label{eq:phi}
\wt F(n) \le d\, \wt
F\left(\left\lfloor\left(\textstyle\frac{d-1}{d^2}+\epsilon\right)n\right\rfloor\right)
+d\log d + \log(2n)
\end{equation}
for all sufficiently large $n$. Roughly speaking, the multiplication of the
argument by $\textstyle\frac{d^2}{d-1}>d$ leads to the multiplication of the
value of $\wt F$ by at most $d$.

We shall consider the partial order $\prq$ on the set of positive functions on
$\mathbb R_+$ defined by $f_1\prq f_2$ if $f_1(t) \le Cf_2(at)$ for certain
constants $a,C>0$, and say that two functions $f_1,f_2$ are \emph{equivalent}
(written $f_1\sim f_2$) if $f_1\prq f_2$ and $f_2\prq f_1$. Inequality
\eqref{eq:phi} implies then

\begin{proposition} \label{pr:entr}
For any $\varepsilon>0$
$$
\wt F(n)\prq n^{\alpha+\varepsilon} \;,
$$
where
\begin{equation} \label{eq:al}
\alpha=\frac{\log d}{\log \textstyle\frac{d^2}{d-1}} < 1 \;.
\end{equation}
\end{proposition}

Recall that the \emph{return profile} $\rho_\mu(n)=\mu_{2n}(e)$ of a symmetric
probability measure $\mu$ on a countable group $G$ is defined as the sequence
of return probabilities to the identity at even times. If the group is
finitely generated then the return profiles of any two symmetric finitely
supported non-degenerate measures $\mu_1,\mu_2$ are equivalent in the sense of
the above definition \cite{Pittet-Saloff99a}. Therefore, one can talk about
(the equivalence class of) the return profile $\rho_G$ of a finitely generated
group $G$ irrespectively of a concrete random walk on this group.

The \emph{isoperimetric profile} of a graph $\Gamma$ is defined as
$$
I_\Gamma (n) = \min \{|V| : |\partial V|/|V| \le 1/n\} \;,
$$
where $\partial V\subset V$ denotes the boundary of a finite vertex subset
$V\subset\Gamma$. In the same way as with the return profiles (actually, it is
much easier to see in this case), the isoperimetric profiles of the Cayley
graphs of a given finitely generated group $G$ corresponding to different
choices of generating sets are all pairwise equivalent, so that one can talk
about (the equivalence class of) the isoperimetric profile $I_G$ of a finitely
generated group $G$.

\begin{theorem} \label{th:profile}
The return and the isoperimetric profiles, respectively, of the Mother group
$\mthr=\mthr(X)$ with $|X|=d$ satisfy, for any $\varepsilon>0$, the relations
$$
\rho_\mthr(n) \succcurlyeq \exp\left(-n^{\alpha+\varepsilon}\right) \qquad
\text{and} \qquad I_\mthr(n) \prq
\exp\left(n^{\frac{2\alpha}{1-\alpha}+\varepsilon}\right) \;,
$$
where $\alpha$ is given by formula \eqref{eq:al}.
\end{theorem}

\begin{proof}
\propref{pr:entr} in combination with the well-known inequality
$\wt\mu_{2n}(e) \ge \exp(-2H(\wt\mu^n))$ immediately implies the lower
estimate for the return profile. By the general Nash inequality machinery (see
\cite{Grigoryan94,Coulhon96} or a later exposition in \cite[Corollary
14.5(b)]{Woess00}) it leads to the upper estimate for the isoperimetric
profile.
\end{proof}

\begin{rem}
We emphasize that our argument provides an upper estimate for the
isoperimetric profile of the Mother groups without producing explicit F\o lner
sets. Finding them should apparently precede any work on establishing the
precise isoperimetric profiles for these groups.
\end{rem}

\begin{rem}
It is interesting to compare the estimates from \thmref{th:profile} with the
precise return and isoperimetric profiles of the \emph{lamplighter groups}
$\LL_k = \Z/2\wr\Z^k = (\Z/2)^{\Z^k}\rtimes\Z^k$,
$$
\rho_{\LL_k}(n) \sim \exp\left(-n^{\alpha_k}\right) \qquad \text{and} \qquad
I_{\LL_k}(n) \sim \exp\left(n^{\frac{2\alpha_k}{1-\alpha_k}}\right) \;,
$$
where $\alpha_k=k/(k+2)$, which were found in \cite{Pittet-Saloff99a} (also
see \cite{Erschler06}) and \cite{Erschler03}, respectively.
\end{rem}

We shall now combine \thmref{th:profile} with \thmref{thm:reduc} in order to
obtain similar estimates for an arbitrary finitely generated subgroup $G$ of
$\bded(X)$. Let us first notice that the return profile
$\rho_{\mthr(X)^d}=\rho^d_{\mthr(X)}$ of the $d$-th power of the Mother group
also satisfies the inequality from \thmref{th:profile}. Since the return
profile does not change when passing to a finite extension (see
\cite{Pittet-Saloff99a}), the return profile of the wreath product
$\mthr(X)\wr\sym{X}$ satisfies this inequality as well. Further, by the
monotonicity of the return profile under passing to subgroups
\cite{Pittet-Saloff00}, the same inequality from \thmref{th:profile} is also
satisfied for the return profile of an arbitrary finitely generated subgroup
of $\mthr(X)\wr\sym{X}$. The corresponding inequality for the isoperimetric
profile follows from the inequality for the return profile in the same way as
in the proof of \thmref{th:profile}. \thmref{thm:reduc} then implies

\begin{cor}
Let $G$ be a finitely generated subgroup of $\bded(X)$, and let $N=N(G)$ be as
in \thmref{thm:reduc}. Then the return and the isoperimetric profiles,
respectively, of the group $G$ satisfy, for any $\varepsilon>0$, the relations
$$
\rho_G(n) \succcurlyeq \exp\left(-n^{\alpha+\varepsilon}\right) \qquad
\text{and} \qquad I_G(n) \prq
\exp\left(n^{\frac{2\alpha}{1-\alpha}+\varepsilon}\right) \;,
$$
where
$$
\alpha=\frac{\log d^N}{\log \textstyle\frac{d^{2N}}{d^N-1}} < 1 \;.
$$
\end{cor}

\appendix
\section{Entropy inequalities} \label{sec:app}

The \emph{entropy} of a discrete probability distribution $p=(p_i)$ is defined
as
$$
H(p) = - \sum p_i \log p_i \;,
$$
and it satisfies the inequality
$$
H(p) \le \log |\supp p\,|
$$
if $p$ has finite support.

Although all the properties of the entropy which we need
(Scholium~\ref{sch:A0}, \lemref{lem:A1} and
\lemref{lem:A2}) could in principle be deduced just from the definition above,
it is more convenient to adopt a more general point of view and argue in terms
of the \emph{entropy of measurable partitions}. See \cite{Rohlin67} for all
the background notions and definitions.

Let $(X,m)$ be a probability measure space, and $\xi=\{\xi_i\}$ be its
countable \emph{measurable partition}, so that $X=\bigcup_i \xi_i$ is a
disjoint union of the measurable \emph{elements} $\xi_i$ of the partition
$\xi$. We shall denote by $\xi(x)$ the element of $\xi$ which contains a point
$x\in X$, and put
$$
m(x;\xi) = m(\xi(x)) \;.
$$
Then the \emph{entropy} of the partition $\xi$ is defined as the
entropy of the distribution of measures of its elements, i.e.,
$$
H(\xi) = - \sum_{C\in\xi} m (C) \log m(C) = - \int_X \log
m(x;\xi)\,dm(x) \;.
$$
The entropy of partitions is monotone in the sense that if $\xi'$ is another
partition finer than $\xi$, i.e., its elements are smaller:
$$
\xi'(x)\subset\xi(x) \text{\ for all\ }x\in X,
$$
then
\begin{equation} \label{eq:Hmonot}
H(\xi') \ge H(\xi) \;.
\end{equation}

Given a measurable subset $C\subset X$ denote by $m_C$ the corresponding
\emph{conditional measure}, i.e., the normalized restriction of the measure
$m$ to $C$, and let $\xi_C$ denote the \emph{trace} of the partition $\xi$ on
the space $(C,m_C)$, i.e., $\xi_C(x) = \xi(x) \cap C$ for any $x\in C$.

If $\zeta$ is another countable partition, set
$$
m(x;\xi|\zeta) = m_{\zeta(x)} (x; \xi_{\zeta(x)}) = m( \xi(x) \cap
\zeta(x) ) / m(\zeta(x)) \;.
$$
Then the (mean) conditional entropy of $\xi$ with respect to
$\zeta$ is defined as the weighted average of the entropies of the
traces of $\xi$ on the elements of $\zeta$:
$$
H(\xi|\zeta) = \sum_{C\in\zeta} m(C) H(\xi_C) = - \int_X \log
m(x;\xi|\zeta)\,dm(x) \;.
$$
The conditional entropy has the property that
$$
H(\xi|\zeta) \le H(\xi) \;,
$$
and it satisfies the identity
$$
H(\xi|\zeta) + H(\zeta) = H(\xi \vee \zeta) \;,
$$
where $\xi\vee\zeta$ is the \emph{join} of the partitions $\xi$
and $\zeta$, i.e.,
$$
(\xi\vee\zeta)(x) = \xi(x) \cap \zeta(x) \qquad \text{for all\
}x\in X,
$$
so that in view of \eqref{eq:Hmonot}
\begin{equation} \label{eq:Hineq}
H(\xi|\zeta) \le H(\xi) \le H(\xi|\zeta) + H(\zeta) = H(\xi \vee \zeta) \le
H(\xi) + H(\zeta) \;.
\end{equation}

We reformulate the right-hand side inequality of~\eqref{eq:Hineq} as a
\begin{sch}\label{sch:A0}
  If $m$ is a probability measure on a countable set $X$, and
  $\pi_i:X\to X_i$ is a family of projections of $X$ which separates
  its points, then the entropies of $m$ and the image measures
  $m_i=\pi_i(m)$ satisfy the inequality
  \[H(m)\le\sum_i H(m_i).\]
\end{sch}

We shall now use inequalities \eqref{eq:Hineq} to obtain the following
properties.

\begin{lem}\label{lem:A1}
If $\{m_i\}_{i\in I}$ is a countable family of probability measures on a
countable set $X$, then for any probability distribution $p=(p_i)$ on the
index set $I$ the entropy of the convex combination $\sum p_i m_i$ satisfies
the inequalities
\begin{equation} \label{eq:entrineq}
\sum_i p_i H(m_i) \le H\left(\sum_i p_i m_i \right) \le \sum_i p_i
H(m_i) + H(p) \;.
\end{equation}
\end{lem}

\begin{proof}
Let us consider the space $I\times X$ with the probability measure
$$
m(i,x) = p(i) m_i(x) \;,
$$
and endow it with the partitions $\xi^I,\xi^X$ with elements $\{i\}\times X$
and $I\times \{x\}$, respectively. Then
$$
H(\xi^X) = H\left(\sum_i p_i m_i \right) \;, \qquad H(\xi^X|\xi^I)
= \sum_i p_i H(m_i) \;, \qquad H(\xi^I) = H(p) \;,
$$
and the claim follows from inequalities \eqref{eq:Hineq}.
\end{proof}

\begin{lem} \label{lem:A2}
For any two probability measures $\mu_1,\mu_2$ on a countable group $G$ the
entropy of their convolution $\mu_1\mu_2$ satisfies the inequalities
\begin{equation} \label{eq:conv}
H(\mu_1), H(\mu_2) \le H(\mu_1\mu_2) \le H(\mu_1) + H(\mu_2) \;.
\end{equation}
\end{lem}

\begin{proof}
By the definition of the convolution, the measure $\mu_1\mu_2$ is the sum of
the translates
$$
\mu_1\mu_2 = \sum_g \mu_2(g)\,  \mu_1 g  \;,
$$
and the inequalities $H(\mu_1) \le H(\mu_1\mu_2) \le H(\mu_1) + H(\mu_2)$
follow from putting $p_g=\mu_2(g)$ and $m_g=\mu_1 g$ in Lemma~\ref{lem:A1}. In
the same way one shows that $H(\mu_2)\le H(\mu_1\mu_2)$.
\end{proof}

\bibliographystyle{amsalpha}
\bibliography{C:/Sorted/MyTEX/mine}

\end{document}